\documentclass[11pt,leqno]{article}
\usepackage[margin=1in]{geometry} 
\usepackage{amssymb,amsfonts,amsmath,bbm,mathrsfs,stmaryrd,mathtools,extarrows}
\usepackage{xcolor}
\usepackage{url}

\usepackage[shortlabels]{enumitem}
\usepackage{tensor}

\usepackage{xr}
\usepackage[colorlinks,
linkcolor=black!75!red,
citecolor=blue,
pdftitle={},
pdfproducer={pdfLaTeX},
pdfpagemode=None,
bookmarksopen=true,
bookmarksnumbered=true,
backref=page]{hyperref}

\usepackage{tikz}
\usetikzlibrary{arrows,calc,decorations.pathreplacing,decorations.markings,intersections,shapes.geometric,through,fit,shapes.symbols,positioning,decorations.pathmorphing}

\usepackage{braket}

\usepackage[amsmath,thmmarks,hyperref]{ntheorem}
\usepackage{cleveref}

\creflabelformat{enumi}{#2#1#3}

\crefname{section}{Section}{Sections}
\crefformat{section}{#2Section~#1#3} 
\Crefformat{section}{#2Section~#1#3} 

\crefname{subsection}{\S}{\S\S}
\AtBeginDocument{%
  \crefformat{subsection}{#2\S#1#3}%
  \Crefformat{subsection}{#2\S#1#3}%
}

\crefname{subsubsection}{\S}{\S\S}
\AtBeginDocument{%
  \crefformat{subsubsection}{#2\S#1#3}%
  \Crefformat{subsubsection}{#2\S#1#3}%
}

\theoremstyle{plain}

\newtheorem{lemma}{Lemma}[section]
\newtheorem{proposition}[lemma]{Proposition}
\newtheorem{corollary}[lemma]{Corollary}
\newtheorem{theorem}[lemma]{Theorem}

\theoremstyle{plain}
\theoremnumbering{Alph}
\newtheorem{theoremN}{Theorem}

\theoremstyle{plain}
\theorembodyfont{\upshape}
\theoremsymbol{\ensuremath{\blacklozenge}}

\newtheorem{definition}[lemma]{Definition}

\newtheorem{remark}[lemma]{Remark}
\newtheorem{remarks}[lemma]{Remarks}
\newtheorem{convention}[lemma]{Convention}
\newtheorem{notation}[lemma]{Notation}
\newtheorem{construction}[lemma]{Construction}
\newtheorem{recollection}[lemma]{Recollection}

\crefname{definition}{definition}{definitions}
\crefformat{definition}{#2definition~#1#3} 
\Crefformat{definition}{#2Definition~#1#3} 

\crefname{ex}{example}{examples}
\crefformat{example}{#2example~#1#3} 
\Crefformat{example}{#2Example~#1#3} 

\crefname{exs}{example}{examples}
\crefformat{examples}{#2example~#1#3} 
\Crefformat{examples}{#2Example~#1#3} 

\crefname{remark}{remark}{remarks}
\crefformat{remark}{#2remark~#1#3} 
\Crefformat{remark}{#2Remark~#1#3} 

\crefname{remarks}{remark}{remarks}
\crefformat{remarks}{#2remark~#1#3} 
\Crefformat{remarks}{#2Remark~#1#3} 

\crefname{convention}{convention}{conventions}
\crefformat{convention}{#2convention~#1#3} 
\Crefformat{convention}{#2Convention~#1#3} 

\crefname{notation}{notation}{notations}
\crefformat{notation}{#2notation~#1#3} 
\Crefformat{notation}{#2Notation~#1#3} 

\crefname{table}{table}{tables}
\crefformat{table}{#2table~#1#3} 
\Crefformat{table}{#2Table~#1#3}

\crefname{lemma}{lemma}{lemmas}
\crefformat{lemma}{#2lemma~#1#3} 
\Crefformat{lemma}{#2Lemma~#1#3} 

\crefname{proposition}{proposition}{propositions}
\crefformat{proposition}{#2proposition~#1#3} 
\Crefformat{proposition}{#2Proposition~#1#3} 

\crefname{propositionN}{proposition}{propositions}
\crefformat{propositionN}{#2proposition~#1#3} 
\Crefformat{propositionN}{#2Proposition~#1#3} 

\crefname{corollary}{corollary}{corollaries}
\crefformat{corollary}{#2corollary~#1#3} 
\Crefformat{corollary}{#2Corollary~#1#3} 

\crefname{corollaryN}{corollary}{corollaries}
\crefformat{corollaryN}{#2corollary~#1#3} 
\Crefformat{corollaryN}{#2Corollary~#1#3} 

\crefname{theorem}{theorem}{theorems}
\crefformat{theorem}{#2theorem~#1#3} 
\Crefformat{theorem}{#2Theorem~#1#3} 

\crefname{theoremN}{theorem}{theorems}
\crefformat{theoremN}{#2theorem~#1#3} 
\Crefformat{theoremN}{#2Theorem~#1#3} 

\crefname{enumi}{}{}
\crefformat{enumi}{#2#1#3}
\Crefformat{enumi}{#2#1#3}

\crefname{assumption}{assumption}{Assumptions}
\crefformat{assumption}{#2assumption~#1#3} 
\Crefformat{assumption}{#2Assumption~#1#3} 

\crefname{construction}{construction}{Constructions}
\crefformat{construction}{#2construction~#1#3} 
\Crefformat{construction}{#2Construction~#1#3} 

\crefname{recollection}{recollection}{Recollections}
\crefformat{recollection}{#2recollection~#1#3} 
\Crefformat{recollection}{#2Recollection~#1#3} 

\crefname{equation}{}{}
\crefformat{equation}{(#2#1#3)} 
\Crefformat{equation}{(#2#1#3)}

\numberwithin{equation}{section}

\theoremstyle{nonumberplain}
\theoremsymbol{\ensuremath{\blacksquare}}

\newtheorem{proof}{Proof}
\newcommand\pf[1]{\newtheorem{#1}{Proof of \Cref{#1}}}

\newcommand\bC{{\mathbb C}}

\newcommand\bP{{\mathbb P}}

\newcommand\bR{{\mathbb R}}

\newcommand\bZ{{\mathbb Z}}

\newcommand\cC{{\mathcal C}}

\newcommand\cE{{\mathcal E}}
\newcommand\cF{{\mathcal F}}

\newcommand\cK{{\mathcal K}}
\newcommand\cL{{\mathcal L}}

\newcommand\cO{{\mathcal O}}

\DeclareMathOperator{\id}{id}
\DeclareMathOperator{\im}{im}

\DeclareMathOperator{\End}{\mathrm{End}}
\DeclareMathOperator{\Hom}{\mathrm{Hom}}
\DeclareMathOperator{\Epi}{\mathrm{Epi}}
\DeclareMathOperator{\Aut}{\mathrm{Aut}}

\DeclareMathOperator{\Proj}{\mathrm{Proj}}
\def\End{\operatorname {End}}
\def\Ext{\operatorname {Ext}}

\def\rk{\operatorname {rk}}

\newcommand{\qedhere}{\mbox{}\hfill\ensuremath{\blacksquare}}

\newcommand{\xrightarrowdbl}[2][]{%
  \xrightarrow[#1]{#2}\mathrel{\mkern-14mu}\rightarrow
}

\title{Symplectic leaves in projective spaces of bundle extensions}
\author{Alexandru Chirvasitu}

\begin{document}

\date{}

\newcommand{\Addresses}{{
  \bigskip
  \footnotesize

  \textsc{Department of Mathematics, University at Buffalo}
  \par\nopagebreak
  \textsc{Buffalo, NY 14260-2900, USA}  
  \par\nopagebreak
  \textit{E-mail address}: \texttt{achirvas@buffalo.edu}


}}

\maketitle

\begin{abstract}
  Fix a stable degree-$n$ rank-$k$ bundle $\mathcal{F}$ on a complex elliptic curve for (coprime) $1\le k<n\ge 3$. We identify the symplectic leaves of the Poisson structure introduced independently by Polishchuk and Feigin-Odesskii on $\mathbb{P}^{n-1}\cong \mathbb{P}\mathrm{Ext}^1(\mathcal{F},\mathcal{O})$ as precisely the loci classifying extensions $0\to \mathcal{O}\to \mathcal{E}\to \mathcal{F}\to 0$ with $\mathcal{E}$ fitting into a fixed isomorphism class, verifying a claim of Feigin-Odesskii. We also classify the bundles $\mathcal{E}$ which do fit into such extensions in geometric / combinatorial terms, involving their Harder-Narasimhan polygons introduced by Shatz.
\end{abstract}


\vspace{.5cm}



\section*{Introduction}

In \cite{fo89_en} Feigin and Odesskii introduced and initiated the study of their far-reaching generalization of Sklyanin's \cite{Skl82,Skl83} eponymous algebras: a family $Q_{n,k}(E,\eta)$ of deformations of the polynomial ring $S:=\bC[x_i,\ 1\le i\le n]$ for each fixed coprime $1\le k<n\ge 3$ and complex elliptic curve $E\cong \bC/\bZ\oplus \bZ\tau$, with $\eta\in E$ as the deformation parameter and recovering the original polynomial ring at $\eta=0\in E$. As familiar from {\it deformation quantization} \cite[\S 1]{weinst_defquant}, this machinery equips $S$ with a {\it Poisson structure}; being homogeneous, that structure will then descend to make $\bP^{n-1}\cong \Proj S$ \cite[Example II.2.5.1]{hrt} into a {\it Poisson manifold} \cite[Definition 1.15]{lgpv_poiss-bk}. For lack of an established term, we refer to this as the {\it deformation (FO) Poisson structure} on $\bP^{n-1}$.

It is a measure of the breadth and reach of the subject, and of its amenability to deep and perhaps unexpected connections, that the deformation Poisson structure can apparently also be recovered in purely geometric terms: \cite[Introduction]{FO98} (and, independently, \cite[\S 2]{pl98}) identifies the ambient $\bP^{n-1}$ with the projectivization
\begin{equation}\label{eq:p0p1}
  \bP H^0(\cF)^*\cong \bP \Ext^1(\cF,\cO)
  \quad
  \left(\text{{\it Serre duality} \cite[Theorem III.7.6]{hrt}}\right)
\end{equation}
for a degree-$n$ rank-$k$ {\it stable} \cite[Definition 5.3.2]{lepot-vb} vector bundle $\cF$ on $E$ and equips that space with what we will refer to as its {\it bundle Poisson structure} (the two sources \cite{FO98,pl98} generalize the construction in different directions). 

There is much literature \cite{HP1,HP2,hp_23,2306.14719v1,MR4197565,2310.05284v1,pol2022,Pym18,Pym17,safr_poisslie} that touches on these ideas in one way or another and, as noted in \cite[discussion following Theorem 4.3]{2310.05284v1}, there does not seem to be much doubt as to the coincidence of the deformation and bundle structures. Nevertheless, the result seems to have been proved formally in print only for $k=1$, as \cite[Theorem 5.2]{HP1}.

The present paper is concerned with the bundle Poisson structure, and specifically with its {\it symplectic leaves} \cite[Proposition 1.3]{weinst_loc}: the maximal connected {\it immersed} \cite[Chapter 2, post Theorem 10]{spiv_dg_1999} submanifolds whose tangent spaces at the points $p\in \bP^{n-1}$ are the images of the respective maps $\left(T_p\bP^{n-1}\right)^*\to T_p\bP^{n-1}$ that constitute the Poisson structure. The classification (\Cref{th:detleaves}) reads:

\begin{theoremN}\label{thn:leaves}
  Fix coprime $1\le k<n\ge 3$ and a stable degree-$n$ rank-$k$ bundle on the complex elliptic curve $E$. 
  
  The symplectic leaves of the bundle Poisson structure of \cite[\S 2]{pl98} on \Cref{eq:p0p1} are the non-empty spaces
  \begin{equation*}
    L(\cE)
    :=
    \text{elements of \Cref{eq:p0p1} classifying non-split exact sequences }
    0\to \cO\to \cE\to \cF\to 0
  \end{equation*}
  via the usual \cite[\S III.5]{gm_halg_2e_2003} correspondence between $\Ext^1$ and extensions.  \qedhere
\end{theoremN}

The eminently believable claim is (essentially) made as part of \cite[Theorem 1]{FO98} and again in passing on \cite[p.67, pre second paragraph]{FO98} in the broader context of bundles on $E$ with reductive/parabolic structure groups, can presumably be recovered for special pairs $(n,k) = (r^2s,rsm-1)$ from \cite[\S 3 and Corollary 5.2]{HP2}, and is proven in this form for $k=1$ in \cite[Theorem 1.1]{00-leaves_xv3}, as part of a more in-depth analysis of the relationship between the {\it homological leaves} \cite[\S 1.2]{00-leaves_xv3} $L(\cE)$ and the {\it secant varieties} \cite[Example 8.5]{H92} of the embedding
\begin{equation*}
  E\subset \bP^{n-1}\cong \bP H^0(\cF)^*,\quad \cF=\text{degree-$n$ line bundle}. 
\end{equation*}

Another phenomenon that transports over from the $k=1$ case is the fact that the homological leaf $L(\cE)$ is the {\it geometric quotient} \cite[Definition 0.6]{mumf-git} of a space of appropriately well-behaved sections of $\cE$ by the (free) action of $\Aut(\cE)$, in the sense familiar from geometric invariant theory. By contrast to the strategy adopted in \cite{00-leaves_xv3}, where the analogous geometric-quotient result \cite[Theorem 5.7]{00-leaves_xv3} {\it precedes} the symplectic-leaf statement \cite[Theorem 6.6]{00-leaves_xv3}, here we {\it deduce} \Cref{th:git} from \Cref{thn:leaves} (\Cref{th:detleaves}):

\begin{theoremN}\label{thn:git}
  In the setting and notation of \Cref{thn:leaves}, the map
  \begin{equation*}
    \left(\cO\xrightarrow{\Phi}\cE\right)
    \xmapsto{\quad}
    \left(\text{class of }0\to \cO\xrightarrow{\Phi} \cE\to \cF\to 0 \right)\in \bP^{n-1}\cong \bP\Ext^1(\cF,\cO)
  \end{equation*}
  realizes the homological leaf $L(\cE)$ as the geometric quotient of the space $\Gamma(\cE)_s\subset \Gamma(E)$ of {\it stable} sections (\Cref{def:stabsect}) by the natural $\Aut(\cE)$-action.  \qedhere
\end{theoremN}

At that point it will have become natural to ask which homological leaves $L(\cE)$ are, in fact, non-empty; \Cref{th:whichmiddle} addresses this.

\begin{theoremN}
  Consider coprime $1\le k<n\ge 3$, $E$ and $\cF$ as in \Cref{thn:leaves}.
  
  The degree-$n$ rank-$(k+1)$ bundles $\cE$ fitting into non-split exact sequences $0\to \cO\to \cE\to \cF\to 0$ are precisely those with determinant $\det\cE=\det\cF$ whose {\it Harder-Narasimhan decomposition} \cite[Proposition 5.4.2]{lepot-vb}
  \begin{equation*}
    \cE\cong \cE_1\oplus \cE_2\oplus \cdots\oplus \cE_s
    ,\quad
    \frac{\deg\c E_1}{\rk \cE_1}
    >
    \cdots
    >
    \frac{\deg\c E_s}{\rk \cE_s}
  \end{equation*}
  is such that the points
  \begin{equation*}
    P_i:=
    (\rk\cE_1+\cdots+\rk \cE_{i},\ \deg\cE_1+\cdots+\deg\cE_{i})
    ,\quad
    1\le i\le s-1    
  \end{equation*}
  lie strictly inside the triangle with vertices $(0,0)$, $(k+1,n)$ and $(k,n)$.  \qedhere 
\end{theoremN}

\subsection*{Acknowledgements}

I am grateful for numerous illuminating comments from R. Kanda, M. Matviichuk, S. P. Smith and B. Pym. 

The work was partially supported by NSF grant DMS-2001128. 


\section{Homological and symplectic leaves}\label{se:ps}

The phrase {\it vector bundle} is the customary one in algebraic geometry (\cite[Exercise II.5.18]{hrt}, \cite[\S B.3]{Fulton-2nd-ed-98}), and we will frequently make the passage between these and their corresponding (finite-rank) locally free sheaves seamlessly and tacitly. 

The present discussion uses the conventions of \cite{pl98}, but we use the symbols `$n$' for degrees and `$k$' for ranks, so as to be in agreement with the notation of \cite{FO98}. Having fixed degrees $n_i$ and ranks $k_i$ for $i=1,2$, Polishchuk's moduli space $M$ (denoted there by $M_\sigma(n_1,n_2,k_1,k_2)$) parametrizes the morphisms $\Phi:\cK\to \cE$ between bundles of degrees $n_i$ and ranks $k_i$ respectively, satisfying an appropriate stability condition depending on the numbers $\sigma$ or $\tau$. We also write `$\cF$' for the $E$s of \cite{pl98} (and `$\cF_i$' for `$E_i$', etc.) so as to avoid notational clashes with \cite{00-leaves_xv3}.

Now fix a stable vector bundle $\cF$ of degree $n$ and rank $k$ (playing the roles of the $\xi_{n,k}$ in \cite[Introduction]{FO98}). As explained in \cite[$\S$3]{pl98}, since the moduli space $\bP^{n-1}\cong \bP\Ext^1(\cF,\cO)$ considered in \cite{FO98} consists of extensions $0 \to \cO \to \cE \to \cF \to 0$, it can be identified with part of the moduli space $M=M_\sigma(n,0,k+1,1)$.  Specifically $\bP^{n-1}$ parametrizes embeddings $\cO \to \cE$ into rank-$(k+1)$ bundles $\cE$ whose determinant is equal to $\det(\cF)$. This differs from $M$ in that we have fixed the determinants of $\cK$ and $\cE$, and hence have two fewer degrees of freedom.

By \cite[Lemma 3.1]{pl98}, the tangent space to $M$ at a point $\cK\xrightarrow{\Phi} \cE$ can be identified with $H^1(\underline{\End}(\cE,\cK))$, where $\underline{\End}(\cE,\cK)$ is the sheaf of local endomorphisms of $\cE$ that preserve the subbundle $\cK$; i.e., it is the kernel in the sequence
\begin{equation*}
  0 \to \underline{\End}(\cE,\cK) \to \underline{\End}(\cE)  \to  \underline{\Hom}(\cK,\cE/\cK) \to 0.
\end{equation*}

\Cref{le:n+1} confirms the aforementioned numerical intuition of dropping two degrees of freedom. Since in the sequel we will have to refer to the notion of {\it stability} of \cite[p.691]{pl98} for a section $\Phi\in \Gamma(\cE)$ of a vector bundle on $E$, we recall it here.

\begin{definition}\label{def:stabsect}
  Let $\cE$ be a vector bundle of rank $\ge 2$ on the elliptic curve $E$. A section $\cO\xrightarrow{\Phi} \cE$ is {\it stable} if the two following mutually-equivalent (sets of) conditions hold:
  \begin{enumerate}[(a),wide=0pt]

  \item\label{item:def:stabsect-stabquot} $\Phi$ is non-zero (and hence an embedding), non-split, and the quotient $\cE/\Phi(\cO)$ is a stable bundle (so in particular torsion-free).

  \item\label{item:def:stabsect-slopes} Every proper non-zero subbundle $\cE_1\le \cE$ has {\it slope} \cite[Definition 10.20]{muk-invmod}
    \begin{equation}\label{eq:def:stabsect-subb}
      \frac{\deg\cE'}{\rk\cE_1} =: \mu(\cE') < \frac{\deg\cE}{\rk\cE-1},
    \end{equation}
    and also
    \begin{equation}\label{eq:def:stabsect-quot}
      \mu\left(\cE/\cE_1\right) > \frac{\deg\cE}{\rk\cE-1}
    \end{equation}
    for strictly intermediate subbundles $\Phi(\cO)\lneq \cE'\lneq \cE$.
  \end{enumerate}
  We then also refer to the resulting extension
  \begin{equation}\label{eq:origphi}
    \xi
    \quad : \quad
    0\to \cO\xrightarrow{\ \Phi\ } \cE\xrightarrow{\quad} \cF\to 0
  \end{equation}
  as {\it stable}.
\end{definition}

\begin{convention}\label{cv:posdeg}
  We only ever discuss the notions relevant to \Cref{def:stabsect} under the assumption that $\deg(\cE)=\deg(\cE/\Phi(\cO))$ is positive.
  
  The only other possibility, if there are to be any stable sections at all, is for $\cE$ to be the unique extension of $\cO$ by $\cO$ (the bundle $F_2$ of \cite[Theorem 5]{Atiyah}).

  An opportunistic piece of notation, on the subject of such iterated extensions of $\cO$: we write $\tensor[_n]{\cO}{}$ for the one of rank $n$ (i.e. $F_n$ in \cite[Theorem 5]{Atiyah}). 
\end{convention}

\begin{remarks}\label{res:nowh0}
  \begin{enumerate}[(1), wide=0pt]

  \item To verify that \Cref{item:def:stabsect-stabquot} and \Cref{item:def:stabsect-slopes} in \Cref{def:stabsect} are indeed equivalent one considers subbundles $\cE'\le \cE$. These fall into two classes: those containing $\cO\lhook\joinrel\xrightarrow{\Phi}\cE$ and the others, etc.
    
  \item Modulo the usual \cite[Exercise II.5.18]{hrt} sheaf-to-vector-bundle correspondence, the sheaf-language requirement in \Cref{def:stabsect}\Cref{item:def:stabsect-stabquot} that $\cE/\Phi(\cO)$ be a torsion-free sheaf is precisely the bundle-theoretic constraint of \cite[p.691]{pl98} that $\Phi$ be nowhere vanishing.
    
  \item\label{item:genslope} Slopes make sense \cite[Definition 4.6]{BB-vb} for arbitrary coherent sheaves on arbitrary smooth projective curves such as $E$: ranks can be made sense of any number of ways (\cite[Definition 4.4]{BB-vb} or \cite[Exercise II.6.12]{hrt}, say), and one can take the {\it Euler characteristic} (\cite[Definition 4.3]{BB-vb}, \cite[Exercise III.5.1]{hrt})
    \begin{equation*}
      \chi(\cF):=\dim H^0(\cF)-\dim H^1(\cF)
    \end{equation*}
    of $\cF$ as a stand-in for the degree. With that in mind, there is no loss or harm in dropping the local freeness constraint on $\cE'$ in the conditions of \Cref{def:stabsect}\Cref{item:def:stabsect-slopes}.

  \item Degrees also make sense \cite[Exercise II.6.12]{hrt} in the generality of point \Cref{item:genslope}, and on elliptic curves specialize back to Euler characteristics (e.g. by the \cite[Riemann-Roch Formula 10.10]{muk-invmod} for bundles and the additivity of both degrees \cite[Exercise II.6.12]{hrt} and Euler characteristics \cite[Exercise III.5.1]{hrt}).

  \item When $\cF$ is a {\it line} bundle of degree $n\ge 3$, the conditions imposed on the bundle $\cE$ of \cite[\S 3.1]{00-leaves_xv3} are precisely \cite[Proposition 3.4]{00-leaves_xv3} equivalent to the requirement that it fit into a stable extension \Cref{eq:origphi}. 
  \end{enumerate}
\end{remarks}

We single out the following simple remark, as it would in any case otherwise be implicit in much of the ensuing discussion.

\begin{lemma}\label{le:openlocus}
  Having fixed the rank-$(k+1)$ bundle $\cE$, the sections $\cO\xrightarrow{\Phi} \cE$ stable in the sense of \Cref{def:stabsect} constitute an open subspace $\Gamma(\cE)_s\subset \Gamma(\cE)$.
\end{lemma}
\begin{proof}
  Said sections are precisely those non-zero sections for which the resulting quotient $\cE/\cO$ is stable (for the stable bundles on $E$ are uniquely determined by their ranks and determinants \cite[Theorem 10]{Atiyah} and \cite[Appendix A, Fact]{tu}), and the openness claim follows from the openness of the stability condition \cite[Theorem 2.8(B)]{maruy_open}.
\end{proof}

\Cref{le:openlocus} implies in particular that much as expected, if there {\it are} any stable sections at all (for a given $\cE$), ``most'' sections are stable. 

\begin{lemma}\label{le:n+1}
  Let $\cO\xrightarrow{\Phi} \cE$ be a stable embedding in the sense of \Cref{def:stabsect}, so that
  \begin{equation*}
    \cF:=\cE/\cO=\cE/\Phi(\cO)
  \end{equation*}
  is a stable bundle.
  
  The tangent space $T=H^1(\underline{\End}(\cE,\cO))$ to the moduli space $M$ at $\Phi$ has dimension $n+1$.
\end{lemma}
\begin{proof}
  Consider the defining exact sequence
  \begin{equation}\label{eq:def}
    0\to
    \underline{\End}(\cE,\cO)
    \to
    \underline{\End}(\cE)
    \to
    \cF
    \to 0
  \end{equation}
  of $\underline{\End}(\cE,\cO)$. We then have the resulting exact sequence of cohomology spaces, taking the following form:
  \begin{equation}\label{eq:seq}
    0\to \End(\cE)/\End(\cE,\cO)\to \Gamma(\cF)\to T\to H^1(\underline{\End}(\cE))\to 0,
  \end{equation}
  where non-underlined $\End$s and $\Hom$s signify vector spaces as opposed to sheaves of $\cO_E$-modules. We will denote the non-zero terms of this sequence by $S_1$ up to $S_4$, the indices increasing rightward.

  Now, since $\underline{\End}(\cE)\cong \cE\otimes\cE^*$ is self-dual, its $1^{st}$ and $0^{th}$ cohomology spaces have the same dimension. On the other hand, the stability of $\cF$ and the fact that the extension
  \begin{equation}\label{eq:ext}
    0\to \cO\to \cE\to \cF\to 0
  \end{equation}
  does not split shows that
  \begin{equation}\label{eq:endfoscalar}
    \End(\cE,\cO) = \text{scalars}\cong \Bbbk
  \end{equation}
  In conclusion, we have
  \begin{equation*}
    \dim(S_1)-\dim(S_4) = -1. 
  \end{equation*}
  This, together with
  \begin{equation*}
    \dim(S_2) = \dim(\Gamma(\cF)) = n
  \end{equation*}
  and
  \begin{equation*}
    \sum_{i=1}^4(-1)^i\dim(S_i)=0
  \end{equation*}
  ensures that we indeed have $\dim(T)=\dim(S_3)=n+1$, as claimed.
\end{proof}
 
In the context of \Cref{le:n+1}, we will now describe the tangent space to $\bP^{n-1}\cong \bP\Ext^1(\cF,\cO)$ at $\Phi:\cO\to \cE$ as a particular 
$(n-1)$-dimensional subspace of the $(n+1)$-dimensional tangent space $H^1(\underline{\End}(\cE,\cO))$.

There is an epimorphism
\begin{equation}\label{eq:epi2}
  \underline{\End}(\cE,\cO)\to \underline{\End}(\cF) \oplus\underline{\End}(\cO)
\end{equation}
obtained by inducing endomorphisms on the subbundle $\cO$ and the quotient bundle $\cF$ of $\cE$. Because $\cF$ is stable, the right-hand side of \Cref{eq:epi2} has the same $1^{st}$ cohomology as $\cO^{\oplus 2}$:
\begin{itemize}
\item Being stable, $\cF$ is \cite[Lemma 12]{tu} indecomposable and hence (\cite[Theorem 10]{Atiyah}, \cite[Proposition 14]{tu}) of the form
  \begin{equation*}
    \cF\cong \cF(k,n)\otimes \cL
    ,\quad
    k:=\rk\cF,\quad n:=\deg \cF
  \end{equation*}
  for Atiyah's canonical indecomposable bundles $\cF(-,-)$ of \cite[p.1]{tu} and a degree-0 line bundle $\cL$.

\item The degree $n$ and rank $r$ are furthermore \cite[Lemma 30]{tu} coprime. 
  
\item Whence
  \begin{equation}\label{eq:sumofls}
    \begin{aligned}
      \underline{\End}(\cF)
      \cong
      \cF\otimes \cF^*
      &\cong
        \cF(k,n)\otimes \cF(k,n)^*\\
      &\cong \bigoplus \cL
        \quad\text{by \cite[Lemma 22]{Atiyah}},
    \end{aligned}    
  \end{equation}
  the sum ranging over the $k^2$ mutually-non-isomorphic degree-0 line bundles $\cL$ with $\cL^{\otimes r}\cong \cO$.

\item We have
  \begin{equation*}
    H^0(\cL)\cong \{0\}\cong H^1(\cL)
  \end{equation*}
  for all $\cL$ in \Cref{eq:sumofls} except for the single choice $\cL=\cO$, hence the claim. 
\end{itemize}
The last non-zero map in the long exact sequence associated to the cokernel \Cref{eq:epi2} is thus a surjection
\begin{equation}\label{eq:surj-h1}
  H^1(\underline{\End}(\cE,\cO))
  \xrightarrowdbl{\quad}
  H^1(\cO^{\oplus 2})\cong \Bbbk^2. 
\end{equation}

\begin{lemma}\label{le:ker}
  Let $\cO\xrightarrow{\ \Phi\ } \cE$ be a point of the moduli space $M=M_\sigma(n,0,k+1,1)$ (as in \Cref{le:n+1}). The tangent space to $\bP^{n-1}\cong \bP(\Ext^1(\cE/\cO,\cO))$ at $\Phi$ is the kernel of the surjection \Cref{eq:surj-h1}.
\end{lemma}
\begin{proof}
  The two $\Bbbk$ summands of the codomain of \Cref{eq:surj-h1} are the tangent spaces $H^1(\underline{\End}(\cO))$ and $H^1(\underline{\End}(\cF))$ at $\cO$ and $\cF$ of the respective moduli spaces of stable bundles, and each component of the map in \Cref{eq:surj-h1} is the map of tangent spaces obtained by differentiating the morphisms of moduli spaces associating $\cK$ and $\cE/\cK$ respectively to $\Phi:\cK\to \cE$.
4~
  The conclusion follows from the fact that the smaller moduli space $\bP^{n-1}$ that we are interested in is obtained from $M$ by fixing the images of the map of moduli spaces that integrates \Cref{eq:surj-h1}.
\end{proof}

\begin{notation}\label{not:eee}
  Whenever we discuss bundles $\cK\subseteq \cE$ as above, $\cF$ denotes $\cE/\cK$ unless specified otherwise. 
\end{notation}

The kernel from the statement of \Cref{le:ker} is amenable to further explication as follows. First, observe that the sheaf epimorphism \Cref{eq:epi2} fits in general into an exact sequence
\begin{equation}\label{eq:mid}
  0\to \underline{\Hom}(\cF,\cK) \to \underline{\End}(\cE,\cK)\to \underline{\End}(\cF)\oplus\underline{\End}(\cK)\to 0.
\end{equation}
When $\cK:=\cO$ as in \Cref{le:n+1}, the long exact sequence of cohomology groups implies that the kernel in question can be identified with the quotient of the $n$-dimensional space $H^1(\cF^*)\cong H^0(\cF)^*$ by the one-dimensional space
\begin{equation*}
  (\End(\cF)\oplus \End(\cO))/\End(\cE,\cO)\cong (\Bbbk\oplus \Bbbk)/\Bbbk. 
\end{equation*}

The discussion above makes it natural to ask how the Poisson structure on $M$ induces one on the codimension-two submanifold $\bP^{n-1}\subset M$. This is easily determined by examining the Poisson structure on the larger manifold as described on \cite[pp. 689-690]{pl98} (recast in the present setting of $\cK=\cO$, $\cF=\cE/\cO$). 

\begin{construction}\label{con:poissonmap}
  One possible description of the map giving $\bP^{n-1}\cong \bP \Ext^1(\cE/\cO,\cO)$ its Poisson structure is as follows.

  \begin{enumerate}[(1)]
  \item\label{item:1} Consider the global-section map
    \begin{equation}\label{eq:eoast2e}
      \Gamma(\underline{\End}(\cE,\cO)^*)\to \Gamma(\cF)
    \end{equation}
    attached (with $\cK:=\cO$) to the canonical morphism
    \begin{equation}\label{eq:eoast2e-pre}
      \underline{\End}(\cE,\cO)^*
      \cong
      \underline{\End}(\cE,\cK)^*
      \to
      \underline{\Hom}(\cE/\cK,\cK)^*
      \cong
      \underline{\Hom}(\cF,\cO)^*
      \cong
      \cF
    \end{equation}
    dual to the injection
    \begin{equation*}
      \underline{\Hom}(\cE/\cK,\cK)\to \underline{\End}(\cE,\cK)
    \end{equation*}
    appearing as the leftmost (non-zero) arrow in \Cref{eq:mid}.

  \item\label{item:2} Compose \Cref{eq:eoast2e} with the connecting map
    \begin{equation}\label{eq:ee1o}
      \Gamma(\cF)\to H^1(\underline{\End}(\cE,\cO))
    \end{equation}
    arising from the exact sequence \Cref{eq:def}.
  \item\label{item:3} Precompose the resulting map with the Serre duality isomorphism
    \begin{equation}\label{eq:serreiso}
      H^1(\underline{\End}(\cE,\cO))^*\cong \Gamma(\underline{\End}(\cE,\cO)^*).
    \end{equation}  
  \end{enumerate}  
\end{construction}

This provides a map $T^*\to T$ for the tangent space
\begin{equation*}
  T=H^1(\underline{\End}(\cE,\cO))
\end{equation*}
which turns out to be the Poisson bivector on $M$.

Because the $1^{st}$ cohomology of the map dual to \Cref{eq:eoast2e-pre} automatically lands in the codimension-two subspace described in \Cref{le:ker}, the Poisson structure factors as
\begin{equation*}
  T^*
  \xrightarrowdbl{\quad}
  S^*
  \xrightarrow{\quad}
  T,
\end{equation*}
where $S\subset T$ is the tangent space to the smaller moduli space $\bP^{n-1}$. The skew symmetry then implies the further factorization
\begin{equation}\label{eq:tsst}
  T^*
  \xrightarrowdbl{\quad}
  S^*
  \xrightarrow{\quad}
  S
  \lhook\joinrel\xrightarrow{\quad}
  T,
\end{equation}
and hence the original Poisson structure on the $(n+1)$-dimensional moduli space $M$ induces one on $\bP^{n-1}$. We cast the target result as the following paraphrase of the above discussion.

\begin{lemma}\label{le:fib}
  In the setting of \Cref{le:n+1} the Poisson structure on $M$ induces one on $\bP^{n-1}\cong \bP\Ext^1(\cF,\cO)$.

  In particular, the symplectic leaves through points of the latter submanifold with respect to the original Poisson structure are contained in $\bP^{n-1}$.
\end{lemma}

\begin{remark}
  We noted above that we have a natural map from the moduli space $M$ to the product of moduli spaces of line bundles and stable rank-$k$ degree-$n$ bundles. \Cref{le:fib} applies to all fibers of this map, i.e. the symplectic leaves of the Poisson structure on $M$ are ``vertical'' (they are contained in fibers of this map).
\end{remark}

Now let $\Phi:\cO\to \cE$ be a point in the moduli space $\bP=\bP\Ext^1(\cF,\cO)$, as discussed above. The tangent space to the symplectic leaf through $\Phi$ at $\Phi$ is precisely the image of the Poisson structure map $S^*\to S$, where $S=T_{\Phi}\bP$ is the tangent space to the moduli space at $\Phi$.

\begin{proposition}\label{pr:mustbeiso}
If the extensions $0 \to \cO \to \cE \to \cF \to 0$ and $0 \to \cO \to \cE' \to \cF \to 0$ belong to the same symplectic leaf of
$\bP \Ext^1(\cF,\cO)$, then $\cE \cong \cE'$. 
\end{proposition}
\begin{proof}
  Recall that we have identified $S$ with a codimension-two subspace of $H^1(\underline{\End}(\cE,\cO))$, and the Poisson structure map is obtained as the restriction of a map that factors through the connecting morphism
\begin{equation*}
  \Gamma(\cF)\to H^1(\underline{\End}(\cE,\cO))
\end{equation*}
associated to the sequence \Cref{eq:def}. It follows in particular that its image is contained in the kernel of the canonical map
\begin{equation}\label{eq:e1oe1}
  H^1(\underline{\End}(\cE,\cO))\to H^1(\underline{\End}(\cE)). 
\end{equation}
Since the right-hand side of this last display is the space of {\it infinitesimal deformations} \cite[Corollary 2.8]{hrt_def} of $\cE$. The symplectic leaves being \cite[\S\S 1.2, 4.1]{clm} the maximal (connected) integral manifolds tangent to vectors in the {\it kernel} of \Cref{eq:e1oe1}, two points on the same symplectic leaf are connectable by curves along which $\cE$ is of constant isomorphism class. It follows that that isomorphism class is globally constant on any given leaf, and we are done. 
\end{proof}

\begin{remark}
  The first map $\Gamma(\underline{\End}(\cE,\cO))\to \Gamma(\cF)$ making up the composite Poisson structure map on $M$ effects an embedding of the $(n-1)$-dimensional space
\begin{equation*}
  \Gamma(\underline{\End}(\cE,\cO))/(\End(\cF)\oplus\End(\cO))
\end{equation*}
into the $n$-dimensional space $\Gamma(\cF)$. The long exact cohomology sequence associated to the dual
\begin{equation*}
  0\to \underline{\End}(\cF)\oplus \underline{\End}(\cO)\to \underline{\End}(\cE,\cO)\to \cF\to 0
\end{equation*}
to \Cref{eq:mid} (with $\cK=\cO$) implies that this hyperplane of $\Gamma(\cF)$ maps through the connecting map $\Gamma(\cF)\to H^1(\underline{\End}(\cE,\cO))$ onto precisely the tangent space at $\Phi$ to the subvariety of $\bP\Ext^1(\cF,\cO)$ consisting of extensions with the same isomorphism type as $\cE$.
\end{remark}

\begin{lemma}\label{le:samehyp}
  The image of the map \Cref{eq:eoast2e} coincides with the hyperplane
  \begin{equation*}
    \im\left(
      \Gamma(\cE)
      \xrightarrow{\quad}
      \Gamma(\cF)
    \right)
  \end{equation*}
  resulting from the original extension \Cref{eq:origphi}.
\end{lemma}
\begin{proof}
  That the image in the statement is indeed a hyperplane follows from a simple dimension count and the long exact sequence
  \begin{equation*}
    0\to
    \left(\Gamma(\cO)\cong \Bbbk\right)
    \to
    \Gamma(\cE)
    \to
    \Gamma(\cF)
    \to
    \left(H^1(\cO)\cong \Bbbk\right)
    \to
    H^1(\cE)
    \to
    H^1(\cF)
    \to 0:
  \end{equation*}
  The last (interesting) term $H^1(\cF)$ vanishes because \cite[Lemma 17]{tu} $\cF$ is assumed stable. Furthermore, $H^1(\cE)$ also vanishes: the stability assumption on $\Phi$ together with the positive-degree requirement of \Cref{cv:posdeg} ensure (\Cref{le:degspos}) that the indecomposable summands of $\cE$ must all have positive degrees, so \cite[Fact in Appendix A and Lemma 17]{tu} apply.
  
  As for that hyperplane being the same as the image of \Cref{eq:eoast2e}, that can easily be seen by fitting both this extension and \Cref{eq:eoe1oe} into
  \begin{equation*}
    \begin{tikzpicture}[>=stealth,auto,baseline=(current  bounding  box.center)]
      \path[anchor=base] 
      (0,0) node (l) {$0$}
      +(2,.5) node (lu) {$\underline{\End}(\cF)\oplus \cO$}
      +(2,-.5) node (ld) {$\cO$}
      +(6,.7) node (ru) {$\underline{\End}(\cE,\cO)^*$}
      +(6,-.7) node (rd) {$\cE$}
      +(8,0) node (r) {$\cF$}
      +(9,0) node (rr) {$0$,}
      ;
      \draw[->] (l) to[bend left=6] node[pos=.5,auto] {$\scriptstyle $} (lu);
      \draw[->] (l) to[bend right=6] node[pos=.5,auto] {$\scriptstyle $} (ld);
      \draw[->] (lu) to[bend left=6] node[pos=.5,auto] {$\scriptstyle $} (ru);
      \draw[->] (ld) to[bend right=6] node[pos=.5,auto] {$\scriptstyle \Phi$} (rd);
      \draw[->] (ru) to[bend left=6] node[pos=.5,auto] {$\scriptstyle $} (r);
      \draw[->] (rd) to[bend right=6] node[pos=.5,auto] {$\scriptstyle $} (r);
      \draw[->] (r) to[bend right=0] node[pos=.5,auto] {$\scriptstyle $} (rr);
      \draw[->] (lu) to[bend left=0] node[pos=.5,auto] {$\scriptstyle $} (ld);
      \draw[->] (ru) to[bend left=0] node[pos=.5,auto] {$\scriptstyle $} (rd);
    \end{tikzpicture}
  \end{equation*}
  where the left-hand vertical map is the obvious projection and the square is a pushout. 
\end{proof}

\begin{proposition}\label{pr:orbmap}
  For a stable embedding $\cO\xrightarrow{\Phi} \cE$, the image of the Poisson map $S^*\to S$ of \Cref{eq:tsst} is canonically isomorphic to $\Gamma(\cE)/\End(\cE)\Phi$. 
\end{proposition}
\begin{proof}
  We need to unpack the various ingredients in \Cref{con:poissonmap} going into the construction of the Poisson bivector. First, because the isomorphism \Cref{eq:serreiso} makes no difference to any ranks and involves no choices, we henceforth disregard it. This leaves the two items in \Cref{item:1} and \Cref{item:2}, to be discussed separately.

  \begin{enumerate}[(1), wide=0pt]
  \item\label{item:eq:eoast2e} {\bf : the map \Cref{eq:eoast2e}.} Dualize \Cref{eq:mid} with $\cK:=\cO$ to 
    \begin{equation}\label{eq:eoe1oe}
      0\to
      \underline{\End}(\cF)\oplus\underline{\End}(\cO)    
      \to
      \underline{\End}(\cE,\cO)^*
      \to
      \cF
      \to 0
    \end{equation}
    (with the self-duality of the sheaves $\underline{\End}$ implicit), whereby \Cref{eq:eoast2e} fits into
    \begin{equation}\label{eq:eoast2e-long}
      0\to
      \Bbbk^2
      \to
      \Gamma(\underline{\End}(\cE,\cO)^*)
      \xrightarrow{\ \text{\Cref{eq:eoast2e}}\ }
      \Gamma(\cF)
      \to
      \Bbbk^2
      \to
      \left(
        H^1(\underline{\End}(\cE,\cO)^*)
        \cong
        \Bbbk
      \right)
      \to 0.
    \end{equation}
    The inline isomorphism (to $\Bbbk$) is due to
    \begin{equation*}
      \begin{aligned}
        H^1(\underline{\End}(\cE,\cO)^*)
        &\cong \End(\cE,\cO)^*
          \quad\text{Serre duality; note that there is no underline!}\\
        &\cong \Bbbk
          \quad\text{by \Cref{eq:endfoscalar}}.
      \end{aligned}
    \end{equation*}
    The image of \Cref{eq:eoast2e} is the hyperplane of \Cref{le:samehyp}. Note also that the quotient
    \begin{equation*}
      \Gamma(\underline{\End}(\cE,\cO)^*) / \Bbbk^2
    \end{equation*}
    that \Cref{eq:eoast2e} embeds via \Cref{eq:eoast2e-long} into $\Gamma(\cF)$ as said hyperplane is canonically identifiable with the $S^*$ of \Cref{eq:tsst}.

  \item\label{item:eq:ee1o} {\bf : the map \Cref{eq:ee1o}.} It is nothing but the map denoted by $\Gamma(\cF)\to T$ in \Cref{eq:seq}, so its image is (naturally identifiable with) the quotient $\Gamma(\cF)/\im\End(\cE)$. Because we are rather interested in the image of the {\it composition} with \Cref{eq:eoast2e} (discussed in \Cref{item:eq:eoast2e} above), note that we have the inclusion
    \begin{equation*}
      \im\End(\cE)\le\text{ hyperplane }\im\Gamma(\cE)\subset \Gamma(\cF). 
    \end{equation*}
    This follows from the commutative diagram
    \begin{equation*}
      \begin{tikzpicture}[>=stealth,auto,baseline=(current  bounding  box.center)]
        \path[anchor=base] 
        (0,0) node (l) {$0$}
        +(2,.5) node (lu) {$\underline{\End}(\cE,\cO)$}
        +(2,-.5) node (ld) {$\cO$}
        +(6,.7) node (ru) {$\cE\otimes \cE^*$}
        +(6,-.7) node (rd) {$\cE$}
        +(8,0) node (r) {$\cF$}
        +(9,0) node (rr) {$0$,}
        ;
        \draw[->] (l) to[bend left=6] node[pos=.5,auto] {$\scriptstyle $} (lu);
        \draw[->] (l) to[bend right=6] node[pos=.5,auto] {$\scriptstyle $} (ld);
        \draw[->] (lu) to[bend left=6] node[pos=.5,auto] {$\scriptstyle $} (ru);
        \draw[->] (ld) to[bend right=6] node[pos=.5,auto,swap] {$\scriptstyle \Phi$} (rd);
        \draw[->] (ru) to[bend left=6] node[pos=.5,auto] {$\scriptstyle $} (r);
        \draw[->] (rd) to[bend right=6] node[pos=.5,auto] {$\scriptstyle $} (r);
        \draw[->] (r) to[bend right=0] node[pos=.5,auto] {$\scriptstyle $} (rr);
        \draw[->] (lu) to[bend left=0] node[pos=.5,auto] {$\scriptstyle $} (ld);
        \draw[->] (ru) to[bend left=0] node[pos=.5,auto] {$\scriptstyle \id\otimes\Phi^*$} (rd);
      \end{tikzpicture}
    \end{equation*}
    pushing \Cref{eq:def} down to \Cref{eq:origphi}, with the left-hand downward map being restriction of sheafy endomorphisms to $\cO\lhook\joinrel\xrightarrow{\Phi} \cE$. Note also that the global-section morphism induced by the {\it right}-hand vertical map is nothing but 
    \begin{equation}\label{eq:ppf}
      \End(\cE)\ni\psi\xmapsto{\quad}\psi\circ\Phi\in\Gamma(\cE).
    \end{equation}
    Having thus identified the image of $S^*\to S$ with
    \begin{equation*}
      \Gamma(\cE)/\im \End(\cE) = \Gamma(\cE)/\End(\cE)\Phi,
    \end{equation*}
    we are done. 
  \end{enumerate}
\end{proof}

\begin{remark}\label{re:pol-end}
  Compare \Cref{pr:orbmap} to \cite[Lemma 3.2]{pl98}, which fits the kernel of the overall composition \Cref{eq:tsst} as the middle term in an extension
  \begin{equation*}
    0\xrightarrow{}
    \End(\cK)\oplus \End(F)
    \xrightarrow{\quad}
    \bullet
    \xrightarrow{\quad}
    \End(\cE)/\End(\cE,\cK)
    \xrightarrow{}0.
  \end{equation*}
  In the present case $\cK=\cO$ the leftmost term
  \begin{equation*}
    \End(\cK)\oplus \End(F)\cong \Bbbk^2
  \end{equation*}
  is precisely what the passage from $T^*$ to $S^*$ (effected in focusing on the middle map $S^*\xrightarrow{}S$ in \Cref{eq:tsst}) annihilates, so \cite[Lemma 3.2]{pl98} can be recast as saying that the kernel of the Poisson structure of interest here is identifiable with $\End(\cE)/\End(\cE,\cO)$. That quotient, in turn, is nothing but the image of \Cref{eq:ppf}. The two results thus differ in emphasis but coincide in substance. 
\end{remark}

Consider the following generalization of \cite[Lemma 3.10]{00-leaves_xv3}; the argument is a slight reworking of the latter's proof. 

\begin{lemma}\label{le:autactsfree}
  For a stable section $\Phi\in\Gamma(\cE)$ as in \Cref{eq:origphi} the $\End(\cE_1)$-module $\End(\cE_1)\Phi$ is free.

  In particular, $\Aut(\cE)$ acts freely on the open subset $\Gamma(\cE)_s\subset \Gamma(\cE)$ of \Cref{le:openlocus}. 
\end{lemma}
\begin{proof}
  The claim is that the upper left-hand map in the composition
  \begin{equation*}
    \begin{tikzpicture}[auto,baseline=(current  bounding  box.center)]
      \path[anchor=base] 
      (0,0) node (l) {$\End(\cE)$}
      +(4,.5) node (u) {$\Gamma(\cE)$}
      +(6,0) node (r) {$\Gamma(\cF)$}
      ;
      \draw[->] (l) to[bend left=6] node[pos=.7,auto] {$\scriptstyle \psi\xmapsto{\quad}\psi\circ\Phi$} (u);
      \draw[->] (u) to[bend left=6] node[pos=.5,auto] {$\scriptstyle $} (r);
      \draw[->] (l) to[bend right=6] node[pos=.5,auto,swap] {$\scriptstyle $} (r);
    \end{tikzpicture}
  \end{equation*}
  is injective. The long exact cohomology sequence of \Cref{eq:def} shows that the kernel of the composition as a whole is $\End(\cE,\cO) \cong \Bbbk$ (the scalars: \Cref{eq:endfoscalar}). Since of course no non-zero scalars annihilate $\Phi$, we are done.
\end{proof}

\begin{theorem}\label{th:detleaves}
  Let $\cF$ be a degree-$n$ rank-$k$ stable bundle on $E$. 
  
  The symplectic leaves of the Poisson structure on $\bP^{n-1}=\bP\Ext^1(\cF,\cO)$ obtained by varying the maps $S^*\to S$ of \Cref{eq:tsst} are precisely the loci consisting of extensions \Cref{eq:origphi} with constant middle-term isomorphism class.
\end{theorem}
\begin{proof}
  In one direction, \Cref{pr:mustbeiso} ensures that the symplectic leaves are {\it contained} in said loci. Conversely, fix $\cE$, and consider the stable locus $U:=\Gamma(\cE)_s\subset \Gamma(\cE)$ provided by \Cref{le:openlocus} along with the map
  \begin{equation}\label{eq:psifromu}
    U\ni \Phi
    \xmapsto{\quad\Psi\quad}
    \Psi(\Phi):=
    \left(
      \text{resulting extension \Cref{eq:origphi} housing $\Phi$}
    \right)
    \in \bP^{n-1}.
  \end{equation}
  

  By \Cref{pr:orbmap}, at each $\Phi\in U$ the differential
  \begin{equation*}
    T_{\Phi}U
    \xrightarrow{\quad d_{\Phi}\Psi\quad}
    T_{\Psi(\Phi)}\bP^{n-1}
  \end{equation*}
  corestricts to a surjection onto the tangent space to the symplectic leaf through $\Psi(\Phi)$, so the image $\Psi(U)$ of the (connected, by \Cref{le:openlocus}!) manifold $U$ must be contained in a single symplectic leaf.
\end{proof}

\begin{remark}\label{re:postleaves}
  There is a claim made on \cite[p.67]{FO98}, to the effect that for a parabolic subgroup $P\le G$ of a semisimple (presumably complex, linear algebraic) group $G$ the leaves of a certain Poisson structure on the moduli space of $P$-bundles on $E$ are exactly the preimages of the structure-group extension map to $G$-bundles.

  \cite[p.690]{pl98} notes that the discussion in \cite[\S 3]{pl98} recovers that on \cite[p.67]{FO98}, presumably for $G=GL(\rk \cE)$ (extending the setup to {\it reductive} groups) and $P\le G$ a {\it maximal} parabolic, i.e. corresponding to the shortest flags. \Cref{th:detleaves} is then essentially a verification of the aforementioned claim in \cite{FO98} regarding symplectic leaves in this specific instance.
\end{remark}


\begin{definition}\label{def:hleaff1}
  The {\it homological leaf} $L(\cE)$ is the image of the morphism
  \begin{equation}\label{eq:psimor}
    \Gamma(\cE)_s
    \xrightarrow{\quad\Psi=\Psi_{\cE}\quad}
    \bP\Ext^1(\cF,\cO)\cong \bP^{n-1}
    ,\quad
    n:=\deg \cF
  \end{equation}
  of \Cref{eq:psifromu}, sending a stable section $\Phi\in\Gamma(\cE)_s$ to the class of the corresponding extension \Cref{eq:origphi}. 
\end{definition}

For rank-2 $\cE$ \cite[Theorem 4.7]{00-leaves_xv3} proves the local closure of $L(\cE)\subset \bP^{n-1}$ by directly identifying them with subschemes of secant varieties. There are more direct routes to that conclusion.

\begin{theorem}\label{th:isloccl}
  The homological leaves $L(\cE)\subset \bP^{n-1}\cong \bP\Ext^1(\cF,\cO)$ are locally closed smooth subvarieties. 
\end{theorem}
\begin{proof}
  We address the two substantive claims in turn.
  
  \begin{enumerate}[(I), wide=0pt]
  \item\label{item:lclos} {\bf : Local closure.} The homological leaves are in any case {\it constructible} (i.e. \cite[Exercise II.3.18]{hrt} finite unions of locally closed subsets) by Chevalley's theorem \cite[Exercise II.3.19]{hrt}. Consider the composition
    \begin{equation}\label{eq:psicoext}
      \Gamma(\cE)_s=:U
      \xrightarrow{\quad\Psi\quad}
      \Psi(U)
      \lhook\joinrel\xrightarrow{\quad}
      \overline{\Psi(U)}
    \end{equation}
    (the closure is unambiguous \cite[Proposition 7]{serre-gaga}: the same in both the Zariski and ordinary topologies). 

    
    By \Cref{le:autactsfree} the fibers of \Cref{eq:psicoext} over every point in its image all have dimension $\dim\Aut(\cE)$, so the equidimensionality condition of \cite[Remark 1.6]{parus_constr} is met. That remark then shows that \Cref{eq:psicoext} is open (it makes no difference which topology \cite[Lemma 2.1]{parus_constr}), meaning precisely the conclusion: $\Psi(U)$ is open in its closure.
    
    
  \item {\bf : Smoothness.} At this stage, given local closure, we can fall back on the general theory of {\it singular foliations} \cite{as_holonomy}, of which symplectic foliations attached to Poisson structures are instances \cite[Example C.23]{clm}.
    
    
    Consider a point $p\in \bP^{n-1}$, and the leaf $L\ni p$ through it. According to \cite[Proposition 1.12]{as_holonomy}, a standard-topology neighborhood $W\ni p$ fibers over a manifold of dimension $(n-1)-\dim L$, with connected fibers contained in leaves. Working only over a neighborhood of $p$ sufficiently small to ensure the closure of $L$ (possible, by local closure) this recovers $L$ as the preimage of a point through a submersion. $L$ must \cite[Corollary 5.13]{lee2013introduction} thus be an {\it embedded} submanifold \cite[\S 5, pp.98-99]{lee2013introduction} and hence smooth as a complex manifold. But then it is also a smooth complex variety \cite[\S XII, Proposition 3.1(iv)]{sga1}, and we are done.
  \end{enumerate} 
\end{proof}


\begin{remarks}\label{res:closedandimm}
  \begin{enumerate}[(1),wide=0pt]
  \item Even given the local closure and immersed-submanifold status of the leaves, some further reference to the specifics of the situation was necessary in the proof of the smoothness claim of \Cref{th:isloccl}: the {\it figure-eight curve} \cite[Example 4.19]{lee2013introduction} is a closed, {\it immersed} \cite[preceding Proposition 5.18]{lee2013introduction} but not embedded submanifold of $\bR^2$, and plainly not smooth as a subset of the plane. 
    
  \item The symplectic leaves of a holomorphic structure are a good deal more than immersed complex submanifolds (though they certainly are that, their tangent spaces being invariant \cite[Proposition 2.11]{lgsx_holpoisman} under the {\it almost complex structure} \cite[Definition 12.4]{dasilva_symplec} of the ambient complex manifold). Every leaf is {\it initial} \cite[Definition C.8]{clm} (or {\it weakly embedded} \cite[following Corollary 5.30]{lee2013introduction}): given a factorization
    \begin{equation*}
      \begin{tikzpicture}[>=stealth,auto,baseline=(current  bounding  box.center)]
        \path[anchor=base] 
        (0,0) node (l) {$N$}
        +(2,.5) node (u) {$L$}
        +(4,0) node (r) {$M$}
        ;
        \draw[->] (l) to[bend left=6] node[pos=.5,auto] {$\scriptstyle $} (u);
        \draw[right hook->] (u) to[bend left=6] node[pos=.5,auto] {$\scriptstyle $} (r);
        \draw[->] (l) to[bend right=6] node[pos=.5,auto,swap] {$\scriptstyle $} (r);
      \end{tikzpicture}
    \end{equation*}
    of a complex-manifold morphism $N\to M$ through a leaf inclusion $L\subseteq M$, the upper left-hand map is also a complex-manifold morphism.

    This follows from the analogue \cite[Theorem C.20]{clm} for plain (smooth) Poisson structures, together with the fact that the maps of complex manifolds are precisely the smooth-manifold morphisms which intertwine the almost complex structures. 
  \end{enumerate}  
\end{remarks}


\begin{corollary}\label{cor:smff}
  The morphism \Cref{eq:psimor} is smooth and faithfully flat. 
\end{corollary}
\begin{proof}
  Being a morphism between smooth varieties (by \Cref{th:isloccl}) as well as a {\it submersion} (i.e. \cite[\S 4, p.77]{lee2013introduction} with surjective differentials) by \Cref{pr:orbmap}, it is smooth \cite[Proposition III.10.4]{hrt}. Flatness is a consequence of smoothness as a matter of {\it definition} \cite[\S III.10]{hrt}, whence faithful flatness = flatness + surjectivity \cite[\S 6.7.8]{ega1}. 
\end{proof}

The selfsame map \Cref{eq:psimor} is a good deal more though. Recall the notion of {\it geometric quotient} by a linear algebraic group of \cite[Definition 0.6]{mumf-git}.

\begin{theorem}\label{th:git}
  The map \Cref{eq:psimor} realizes the homological leaf as the geometric quotient
  \begin{equation*}
    L(\cE)\cong \Gamma(\cE)_s/\Aut(\cE).
  \end{equation*}
\end{theorem}
\begin{proof}
  We run through the hypotheses of \cite[Chapter II, Proposition 6.6]{borel-LAG-se}. $\Psi$ is of course surjective, and we noted its openness in the proof of \Cref{th:isloccl}. $\Psi$ is also
  \begin{itemize}
  \item an {\it orbit map} \cite[Chapter II, \S 6.3]{borel-LAG-se} (its fibers are precisely the orbits of $\Aut(\cE)$), as follows immediately from the fact that $\cF$ being stable, $\Aut(\cF)\cong \bC^{\times}$ \cite[Definition 10.18 and Corollary 10.25]{muk-invmod};
  \item and {\it separable} \cite[Chapter AG, \S 8.2]{borel-LAG-se} because we are in characteristic zero (so {\it all} field extensions are separable).
  \end{itemize}
  The conclusion follows from the fact that $\Psi$ has irreducible domain (open subscheme $\Gamma(\cE)_s\subset \Gamma(\cE)$ of an affine space) and normal codomain $L(\cE)$ (because smooth: \Cref{th:isloccl}).
\end{proof}

\section{Relevant middle terms}\label{se:midterm}

We made the following remark in passing above, in the course of the proof of \Cref{le:samehyp}.

\begin{lemma}\label{le:degspos}
  The indecomposable summands of a positive-degree vector bundle $\cE$ with non-empty homological leaf $L(\cE)$ themselves have of positive degree.
\end{lemma}
\begin{proof}
  Indeed, were there summand $0\ne \cE'\le_{\oplus}\cE$ of non-positive degree, a complementary summand would violate \Cref{eq:def:stabsect-subb}. 
\end{proof}

It will also be of some interest to understand how the various leaves $L(\cE)$ relate to one another: some lie in the closures of others, and the resulting ordering is bound to be illuminating in the sequel. To place the discussion in its proper context, we begin by noting that the stratification of $\bP^{n-1}\cong \bP\Ext^1(\cF,\cO)$ by the $L(\cE)$ is a refinement of the usual one by {\it type} \cite[\S 7, pp.565-566]{zbMATH03803614}. We recall some of the surrounding machinery. 

\begin{recollection}\label{rec:ab}
  \begin{enumerate}[(1), wide=0pt]
  \item Being of fixed rank and degree (i.e. $1^{st}$ Chern class in the usual, topological sense \cite[\S 14.2]{ms-cc}), the bundles $\cE$ appearing as middle terms in extensions \Cref{eq:origphi} are all $C^{\infty}$-isomorphic (\cite[p.2, Proposition]{zbMATH00975620} argues purely topologically, but the discussion goes through fine in the category of $C^{\infty}$ manifolds).

  \item We are thus in the setting of \cite[\S 7]{zbMATH03803614}, and can identify the {\it holomorphic} (or algebraic, by GAGA \cite[12., pp.19-20]{serre-gaga}) isomorphism classes of $\cE$ with the space $\cC$ of \cite[p.565]{zbMATH03803614}. Each $\cE$ admits a unique {\it Harder-Narasimhan (or HN) filtration} (\cite[Proposition 1.3.9]{zbMATH03506873}, \cite[Theorem 1]{zbMATH03577367}, \cite[(7.1)]{zbMATH03803614}, etc.)
    \begin{equation*}
      0=\cE_{0}\subset \cE_{1}\subset\cdots\subset \cE_{s}=\cE
      ,\quad
      \overline{\cE}_{i}:=\cE_{i}/\cE_{i-1}\text{ semistable}
    \end{equation*}
    with the slopes of the $\overline{\cE}_{i}$ strictly decreasing.
    
    Incidentally, since we are working on an elliptic curve, that filtration splits:
    \begin{equation*}
      \cE = \bigoplus_i \overline{\cE}_{i}.
    \end{equation*}
    This follows (say) from the semistability of indecomposable bundles \cite[Appendix A, Fact]{tu}, and justifies the term {\it Harder-Narasimhan (or HN) splitting} we will occasionally use. Other cognates of the term will occasionally appear: {\it HN summands} or perhaps {\it components}, etc.
    
  \item The ({\it HN} or {\it Harder-Narasimhan}) {\it type} $\nu$ of $\cE$ is defined to be the tuple
    \begin{equation}\label{eq:type}
      \nu=\nu(\cE) = \left((k_i,n_i),\ 1\le i\le s\right)
      ,\quad k_i:=\rk \overline{\cE_{i}}
      ,\quad n_i:=\deg \overline{\cE_{i}}.
    \end{equation}
    $\cC$ then admits \cite[Theorem 7.14]{zbMATH03803614} a stratification by
    \begin{equation*}
      \cC_{\nu}:=\left\{\text{classes in $\cC$ of type $\nu$}\right\}
    \end{equation*}
    {\it perfect} in the sense of \cite[\S 1, p.537]{zbMATH03803614}. This naturally lifts to
    \begin{equation*}
      \begin{aligned}
        \bP^{n-1}&=\bP\Ext^1(\cF,\cO) = \coprod_{\nu}\bP \Ext^1(\cF,\cO)_{\nu} =: \bP^{n-1}_{\nu},\\
        \bP \Ext^1(\cF,\cO)_{\nu} &:= \left\{\text{extensions \Cref{eq:origphi} with type-$\nu$ middle term}\right\}.
      \end{aligned}    
    \end{equation*}
    The pieces $\bP^{n-1}_{\nu}$ are again locally closed (see e.g. the discussion on \cite[p.611, preceding Proposition 15.4]{zbMATH03803614} or \cite[\S 11.1, concluding Remark or Corollary 15.4.3]{lepot-vb}), and offer a coarser partition than that into the symplectic leaves of \Cref{th:detleaves}.
    
  \item As for closures, \cite[Theorem 3]{zbMATH03577367} (also recalled as \cite[(7.8)]{zbMATH03803614}) gives useful information, recast in the present setting as 
    \begin{equation*}
      \overline{\bP^{n-1}_{\nu}}\subseteq \coprod_{\nu'\ge \nu}\bP^{n-1}_{\nu'},
    \end{equation*}
    where the ordering $\nu\le \nu'$ is best understood geometrically: the {\it Harder-Narasimhan polygon (HNP for short)} \cite[p.173]{zbMATH03577367}
    \begin{equation}\label{eq:hnp}
      HNP(\nu)=HNP(\cE):=\text{convex hull of }(0,0),\ (k_1,n_1),\ (k_1+k_2,n_1+n_2),\ \cdots
    \end{equation}
    lies within the analogous convex polygon $HNP(\nu')$.   
  \end{enumerate}  
\end{recollection}

The following picture is illustrative of HN polygons (see also \cite[diagram on p.173]{zbMATH03577367}):

\begin{equation*}
  \begin{tikzpicture}[>=stealth,auto,baseline=(current  bounding  box.center)]
    \path[anchor=base] 
    (0,0) node[circle,inner sep=1pt,draw,fill] (0) {$$}
    +(6,0) node (x) {$$}
    +(0,6) node (y) {$$}
    +(1,2) node[circle,inner sep=1pt,draw,fill] (1) {$$}
    +(2,3) node[circle,inner sep=1pt,draw,fill] (2) {$$}
    +(3,3.5) node[circle,inner sep=1pt,draw,fill,label=above:] (3) {$$}
    +(5,4) node[circle,inner sep=1pt,draw,fill] (4) {$$}
    ;
    \draw[->] (0) to[bend left=0] node[pos=.5,auto,swap] {$\scriptstyle \text{rank}$} (x);
    \draw[->] (0) to[bend left=0] node[pos=.5,auto] {$\scriptstyle \text{degree}$} (y);

    \foreach[evaluate=\n as \nextn using int(\n+1)] \n in {0,...,3}
    {  
      \draw[-] (\n) to[bend left=0] node[pos=.5,auto] {$\scriptstyle $} (\nextn);
    }
    
    
    \draw[-] (0) to[bend left=0] node[pos=.5,auto] {$\scriptstyle $} (4);
  \end{tikzpicture}
\end{equation*}


\begin{remarks}\label{res:fostrat}
  \begin{enumerate}[(1), wide=0pt]
  \item\label{item:res:fostrat-k1} Note the symbol reversal $k\leftrightarrow n$ when comparing the above to \cite[\S 7]{zbMATH03803614}: degrees ($n$ for us) are intended as $y$-coordinates, so that the algebraic geometer's slopes, defined as $\frac{\text{degree}}{\text{rank}}$, truly are geometric slopes. 
    
    Observe also that since we are interested in types of rank-$(k+1)$ middle terms $\cE$, the total width of the Harder-Narasimhan polygons of interest is $k+1$ (rather than $k$).
    
  \item\label{item:res:fostrat-charge} Following \cite[discussion preceding Remark 4.17]{BB-vb}, say, we refer to pairs of the form
    \begin{equation*}
      \begin{aligned}
        \zeta(\cE)  &:=  (\rk \cE,\chi (\cE))
                      \quad\text{for coherent $\cE$, with $\chi=$ Euler characteristic (\Cref{res:nowh0}\Cref{item:genslope})}\\
                    &= (\rk \cE,\deg \cE)
                      \quad\text{for locally free $\cE$}.
      \end{aligned}      
    \end{equation*}
    as {\it charges} (of the respective sheaves). As observed in \cite[paragraph preceding Definition 4.6]{BB-vb}, $\zeta$ is additive over exact sequences if its images are regarded as elements of the free abelian group $(\bZ^2,+)$. 
    
  \item The above discussion of stratifications and polygonal ordering is very much also in the spirit of the material on ``rough strata'' in \cite[\S 1 4.]{FO98}. The roughness again refers to the partition into type strata being coarser than that into isomorphism-class leaves.  
  \end{enumerate}  
\end{remarks}

In working within the framework set out by \Cref{rec:ab}, we write $\Delta_{P,Q,R}$ for the planar triangle with vertices $P$, $Q$ and $R$. Other convenient notation (and terminology):

\begin{itemize}[wide]
\item For $1\le k<n$ set
\begin{equation}\label{eq:sliver}
  \tensor[_{k,n}]{\Delta}{}
  :=
  \Delta_{(0,0),(k+1,n),(k,n)}.
\end{equation}

\item For bundles $\cF$ of charge $(k,n)$ we refer to the ``bottom'' edge of $HNP(\cF)$, connecting the origin $(0,0)$ and $(k,n)$ itself, as the {\it base} of the Harder-Narasimhan polygon.

\item Similarly, the {\it top} of $HNP(\cF)$ is the polygonal line excluding the base. It consists of the {\it top edges}. 
\end{itemize}

\begin{theorem}\label{th:whichmiddle}
  Let $1\le k<n$, and fix a stable bundle $\cF$ of charge $(k,n)$ (so in particular $\gcd(k,n)=1$). 

  The middle terms $\cE$ fitting into stable extensions \Cref{eq:origphi} are precisely the charge-$(k+1,n)$ bundles with
  \begin{itemize}[wide]
  \item $\det \cE = \det \cF$;
  \item and the top edges of $HNP(\cE)$ (excluding the two extreme vertices $(0,0)$ and $(k+1,n)$) strictly contained in the triangle $\tensor[_{k,n}]{\Delta}{}$ of \Cref{eq:sliver}. 
  \end{itemize}
\end{theorem}

The following simple preliminary observation seems more difficult to locate explicitly than its converse (in various forms: \cite[Proposition 4]{zbMATH03577367}, \cite[(7.5)]{zbMATH03803614}, \cite[Proposition 5.3.3]{lepot-vb}, \cite[Lemma 14.1]{Polishchuk-book}, \cite[Exercise 6 following \S 10.4]{muk-invmod}, etc.). We record it here with a proof, for completeness.

\begin{lemma}\label{le:morbetslopes}
  If $\cE$ and $\cE'$ are bundles on an elliptic curve with $\mu(\cE)<\mu(\cE')$ then there are non-zero morphisms $\cE\to \cE'$. 
\end{lemma}
\begin{proof}
  Replace $\cE$ with its lowest-slope summand in the HN splitting and similarly, replace $\cE'$ with its {\it largest}-slope HN summand. The former slope can only decrease and the latter can only increase in the process \cite[\S 2 (B), p.167]{zbMATH03577367}, so the inequality is retained. We can thus assume without loss of generality that $\cE$ and $\cE'$ are semistable.

  The dual $\cE^*$ being semistable \cite[Lemma 10.23]{muk-invmod}, so is the tensor product $\cE'\otimes \cE^*$ \cite[Theorem 10.2.1]{lepot-vb}. Furthermore, that tensor product has slope
  \begin{equation}\label{eq:mue'east}
    \mu(\cE'\otimes \cE^*)
    \xlongequal{\ \text{\cite[Theorem 10.2.1]{lepot-vb}}\ }
    \mu(\cE')+\mu(\cE^*)
    =
    \mu(\cE')-\mu(\cE),
  \end{equation}
  assumed positive. But then
  \begin{equation}\label{eq:dimhomee'}
    \dim\Hom(\cE,\cE')
    =
    \dim\Gamma(\cE'\otimes \cE^*)
    \xlongequal{\ \text{\cite[Lemma 17]{tu}}\ }
    \deg\left(\cE'\otimes \cE^*\right) > 0,
  \end{equation}
  and we are done
\end{proof}

\begin{remarks}\label{res:le:morbetslopes}
  \begin{enumerate}[(1),wide=0pt]
  \item\label{item:bb-dimhom} For {\it indecomposable} sheaves the result is stated in \cite[Summary following Corollary 4.25]{BB-vb} in the more precise form saying that the dimension of $\Hom(\cE,\cE')$ is
    \begin{equation*}
      \deg(\cE')\rk(\cE) - \deg(\cE)\rk(\cE') = \text{\Cref{eq:dimhomee'}}.
    \end{equation*}

  \item\label{item:res:le:morbetslopes:strictmuneeded} The {\it strict} inequality is necessary in \Cref{le:morbetslopes}: non-trivial degree-0 line bundles have no non-zero sections, i.e. admit no non-zero morphisms from $\cO$. 
  \end{enumerate}
\end{remarks}

\pf{th:whichmiddle}
\begin{th:whichmiddle}
  \begin{enumerate}[{}, wide=0pt]
  \item {\bf ($\xRightarrow{}$):} That is, we address the necessity of the two conditions. The determinants are equal because determinants are generally multiplicative in exact sequences and the leftmost (interesting) term in \Cref{eq:origphi} is trivial.

    The strict containment, on the other hand, follows from the stability requirement. Indeed, that constraint demands that the summands in the HN splitting of $\cE$ have slopes smaller (strictly) than $\frac nk$, and they in any case have slopes $\ge \mu(\cE)=\frac{n}{k+1}$. But this means precisely that the edges of $HNP(\cE)$ lie inside $\tensor[_{k,n}]{\Delta}{}$, not touching the upper border of that triangle. 
    
  \item {\bf ($\xLeftarrow{}$):} It will be enough to argue that for $\cE$ meeting the requirements on $HNP(\cE)$, some section
    \begin{equation*}
      \Phi\in H^0(\cE)^{\times}
      :=
      H^0(\cE)\setminus \{0\}
    \end{equation*}
    is stable. Or: we have to show that for some $\Phi$ the quotient $\cF':=\cE/\im \Phi$ is indecomposable (for then that quotient will also be stable \cite[Corollary 14.8]{Polishchuk-book}, having coprime rank and degree $k$ and $n$ respectively). There are a number of issues to address.

    \begin{enumerate}[(I),wide=0pt]
    \item {\bf : Torsion.} That the quotient $\cF'$ is torsion-free for $s$ ranging over a dense open subset of $\Gamma(\cE)^{\times}$ follows as in \cite[proof of Proposition 3.20 and/or Remark 3.21]{00-leaves_xv3}: appending back $0\in\Gamma(\cE)$, the undesirable locus of $\Phi\in \Gamma(\cE)$ (i.e. $\Phi=0$ or such that $\cF'$ {\it does} have torsion) is
      \begin{equation*}
        \bigcup_{z\in E}\Hom(\cO(z),\cE)
        =
        \bigcup_{z\in E}\Gamma(\cE(-z)). 
      \end{equation*}
      That subset of $\Gamma(\cE)$ is closed as in the aforementioned result, and a simple dimension count shows that it must be proper.

    \item\label{item:th:whichmiddle:torsfree} {\bf : $\cF'$ assumed torsion-free.} The vector-bundle quotients of the fixed $\cE$ fall into finitely many HN types \Cref{eq:type} (a familiar {\it boundedness} principle; see e.g. \cite[Example preceding Proposition 5.1.1]{lepot-vb}), so it is enough to argue that for any fixed
      \begin{equation*}
        \nu=\left((k_i,n_i),\ 1\le i\le s\right)
      \end{equation*}
      the space of sections $\Phi\in\Gamma(\cE)^{\times}$ with $\nu\left(\cE/\im \Phi\right) = \nu$ has dimension {\it strictly} smaller than $\dim\Gamma(\cE)=\deg \cE$ (the latter equality follows from \cite[Lemma 17]{tu}, given that the indecomposable summands of $\cE$ are all assumed of positive degree). Equivalently, it will do to prove that for types $\nu$ with at least {\it two} components (i.e. types of decomposable bundles)
      \begin{equation*}
        \Gamma(\cE)_{\nu}:=\left\{\Phi\in \Gamma(\cE)^{\times}\ |\ \nu\left(\cE/\im \Phi\right) = \nu\right\}
      \end{equation*}
      has dimension strictly smaller than that of the corresponding space $\Gamma(\cE)_{\nu_0}$ for the indecomposable type $\nu_0:=((k,n))$. 
      
      Having fixed $\nu$, each of the $s$ fixed-charge indecomposable summands ranges over a parameter space isomorphic to $E$ \cite[Theorem 10]{Atiyah}. On the other hand, for each fixed type-$\nu$ bundle $\cF'$ the sections $\Phi$ with
      \begin{equation}\label{eq:quotisof'}
        \cE/\im \Phi\cong \text{the already-fixed }\cF'
      \end{equation}
      are the isomorphisms of $\cO$ onto the kernels of the surjections $\cE\xrightarrow{}\cF'$. 
      
      The space
      \begin{equation}\label{eq:epief'}
        \left\{\text{epimorphisms }\cE\xrightarrowdbl{}\cF'\right\}=:\Epi(\cE,\cF')\subset \Hom(\cE,\cF')
      \end{equation}
      is Zariski-open by the usual semicontinuity argument, e.g. as in \cite[Lemma 2.6.1]{lepot-vb}: matrices have generically large rank, etc. If we assume it non-empty, \Cref{eq:epief'} will also be dense and hence a scheme of dimension $\dim\Hom(\cE,\cF')$, independent of $\cF'$ (among the possible choices: direct sums of positive-degree indecomposable summands, with $\det\cF'=\det \cE$).

      The kernels of epimorphisms $\cE\xrightarrowdbl{}\cF'$ are unaffected by composing with automorphisms of $\cF'$; if the HN splitting of $\cF'$ consists of $s\ge 2$ summands, then the automorphism group of $\cF'$ is at least $(s+1)$-dimensional: each summand can be scaled independently, and there are non-zero morphisms from summands of smaller slope to those of higher (\Cref{le:morbetslopes}). On the other hand, the only endomorphisms of the indecomposable (hence also stable \cite[Appendix A, Fact]{tu}) $\cF'$ are scalars (\cite[Corollary 5.3.4]{lepot-vb}, \cite[Proposition 10.24]{muk-invmod}, etc.).

      The conclusion follows: only for decomposable types $\nu$ does the dimension $\dim\Aut(\cF')\ge s+1$ {\it more} than counteract the dimension $s$ of the parameter space for the possible $\cF'$.
    \end{enumerate}
  \end{enumerate}
  This completes the proof of the claim. 
\end{th:whichmiddle}

\begin{remark}\label{re:recrk2class}
  For line bundles $\cF$ of degree $n\ge 3$ \Cref{th:whichmiddle} recovers the classification \cite[Proposition 3.4]{00-leaves_xv3} of possible middle terms $\cE$ in \Cref{eq:origphi}. Indeed, in that case $k=1$, so that $k+1=2$. For that reason, specifying a convex polygon $HNP(\cE)$ strictly inside
  \begin{equation*}
    \tensor[_{k,n}]{\Delta}{}
    =
    \Delta_{(0,0),(2,n),(1,n)}
  \end{equation*}
   and with the same base (in the sense of the discussion preceding \Cref{th:whichmiddle}) simply means selecting a $y$-coordinate for the third (i.e. different from $(0,0)$ and $(2,n)$) vertex $(1,m)$. The $m$ ranges over $\left[\frac n2,n\right)$, and for decomposable $\cE$ it is the larger of the two degrees of the line-bundle summands of $\cE$. 
\end{remark}




\addcontentsline{toc}{section}{References}

\def\cprime{$'$}

\Addresses

\end{document}